\documentclass[12pt]{amsart}
\usepackage{amsthm}
\usepackage{a4wide}
\usepackage{array, xcolor}
\usepackage{amsfonts}
\usepackage{amsmath}
\usepackage{bbm}
\usepackage{mathtools}
\usepackage{amssymb}
\usepackage{graphicx}
\usepackage{accents}
\usepackage{verbatim}
\newtheorem{proposition}{Proposition}
\newtheorem{lemma}[proposition]{Lemma}
\newtheorem{corollary}[proposition]{Corollary}
\newtheorem{theorem}[proposition]{Theorem}
\usepackage{tikz}
\usepackage[foot]{amsaddr}
\usetikzlibrary{matrix}

\theoremstyle{definition}

\newtheorem{remark}[proposition]{Remark}
\newtheorem*{remark*}{Remark}

\usepackage{hyperref}
\hypersetup{colorlinks=true, breaklinks=true, urlcolor= blue, linkcolor= black, citecolor=teal,
  pdftitle={Spectrum of the Dirac operator on random hyperbolic surfaces}, 
  pdfauthor={Laura Monk \and Rare\c s Stan}, 
  pdfsubject={}}

\def\RR{{\mathbb R}}
\def\Rn{{\mathbb R^n}}

\def\CC{{\mathbb C}}
\def\HH{{\mathbb{H}}}

\def\g{{\gamma}}
\def\G{{\Gamma}}
\renewcommand{\d}{\, \mathrm{d}}

\DeclareMathOperator{\A}{A}
\DeclareMathOperator{\B}{B}
\DeclareMathOperator{\D}{D}
\DeclareMathOperator{\Da}{|D|}
\newcommand{\fund}{F}
\newcommand{\eig}{\lambda}

\DeclareMathOperator{\psl}{PSL_2(\mathbb{R})}

\DeclareMathOperator{\area}{Area}

\DeclareMathOperator{\End}{End}

\newcommand{\cl}{\operatorname{Cl}}
\newcommand{\so}{\operatorname{SO}}
\newcommand{\spin}{\operatorname{Spin}}

\newcommand{\id}{\operatorname{id}}

\newcommand{\slinear}{\operatorname{SL}}

\begin{document}
\title[]
{Spectral convergence of the Dirac operator \protect\\ on typical hyperbolic surfaces of high genus}
\author{Laura Monk\textsuperscript{1} \and Rare\c s Stan\textsuperscript{2}}

\address[1]{School of Mathematics, University of Bristol, Bristol BS8 1UG, U.K.}
\address[2]{Institute of Mathematics of the Romanian Academy, Bucharest, Romania}

\email{laura.monk@bristol.ac.uk}
\email{rares.stan@imar.ro}

\makeatletter
\@namedef{subjclassname@2020}{\textup{2020} Mathematics Subject Classification}
\makeatother

\begin{abstract}
  In this article, we study the Dirac spectrum of typical hyperbolic surfaces of finite area,
  equipped with a nontrivial spin structure (so that the Dirac spectrum is discrete). For random
  Weil--Petersson surfaces of large genus $g$ with $o(\sqrt{g})$ cusps, we prove convergence of the
  spectral density to the spectral density of the hyperbolic plane, with quantitative error
  estimates. This result implies upper bounds on spectral counting functions and multiplicities, as
  well as a uniform Weyl law, true for typical hyperbolic surfaces equipped with any nontrivial spin
  structure.
\end{abstract}

\date{\today}
\maketitle

\section{Introduction}

\subsection{Setting and motivation}
\label{sec:setting-motivation}

The objective of this article is to provide information on the Dirac spectrum of typical hyperbolic
surfaces of genus $g$ with $k$ cusps, where $g$, $k$ are non-negative integers such that $2g-2+k>0$.
In order to do so, we equip the moduli space of hyperbolic surfaces of signature $(g, k)$ with the
Weil--Petersson probability measure~$\mathbb{P}_{g,k}$. This is a natural model to study typical
hyperbolic surfaces, as illustrated by the rich literature that has developed in the last few years
\cite{AnantharamanMonk,Gong,Hide,HideThomas,LipnowskiWright,Mirzakhani,Monk,WuXue,WuXue2}.  By
\emph{typical}, we mean that we wish to prove properties true with probability going to one in a
certain asymptotic regime.

An example of interesting regime is the \emph{large-scale regime}, i.e. the situation when $g$
and/or $k$ go to infinity. Indeed, by the Gauss--Bonnet formula, the area of any hyperbolic surface
of signature $(g,k)$ is $2\pi(2g-2+k)$. Since $g$ and $k$ are two independent parameters, we can a
priori expect typical surfaces of large genus or large number of cusps to exhibit different
geometric and spectral properties. This has been confirmed by the recent complementary works of Hide
\cite{Hide} and Shen--Wu \cite{ShenWu}, which prove very different behaviours for the Laplacian
spectrum depending on whether $k \ll \sqrt{g}$ or $k \gg \sqrt{g}$.

In this article, we fix a sequence of non-negative integers $(k(g))_{g \geq 2}$ such that
$k(g) = o(\sqrt{g})$, i.e.  $k(g)/\sqrt{g} \rightarrow 0$ as $g \rightarrow + \infty$. This setting
is the genus-dominated regime, and we leave the cusp-dominated regime $k(g) \gg \sqrt{g}$ to further
work. Under this hypothesis, the second author \cite{Monk_thesis} and Le Masson--Sahlsten
\cite{LemassonSahlsten} have proven that there exists a set $\mathcal{A}_{g,k(g)}$ of ``good
hyperbolic surfaces'' of signature $(g,k(g))$ such that:
\begin{itemize}
\item the Weil--Petersson probability of $\mathcal{A}_{g,k(g)}$ goes to one as
  $g \rightarrow + \infty$;
\item the systole of any $X \in \mathcal{A}_{g,k(g)}$ is bounded below by $g^{- 1/24}
  \sqrt{\log(g)}$;
\item elements $X \in \mathcal{A}_{g,k(g)}$ are close to the hyperbolic plane in the sense of
  Benjamini--Schramm, and more precisely, the proportion of points on $X$ of injectivity radius
  smaller than $\frac 16 \log(g)$ is at most $g^{-1/3}$.
\end{itemize}
We prove upper bounds and asymptotics for the Dirac spectrum of hyperbolic surfaces in
$\mathcal{A}_{g,k(g)}$, which we shall now present.

\subsection{Spectral convergence of the Dirac operator}
\label{sec:spectr-dirac-oper}

For a hyperbolic surface $X$ of signature $(g,k)$, a spin structure $\varepsilon$ on $X$, we denote
as $\D$ the Dirac operator on $(X, \varepsilon)$.

While the Laplacian spectrum of a hyperbolic surface with cusps always contains essential spectrum,
equal to $[1/4, + \infty)$, for any $X$, one can pick $\varepsilon$ so that the spectrum of the
Dirac operator $\D$ is discrete \cite{Bar}. We call such spin structures \emph{nontrivial}.  In
that case, for $0 \leq a \leq b$, we denote as $N^{\Da}_{(X,\varepsilon)}(a,b)$ the number of
eigenvalues of the absolute value $\Da$ of the Dirac operator $\D$, once all rigid multiplicities
are removed (see Section \ref{sec:spectr-mult}).

Our main result is the spectral convergence of the Dirac operator on $(X,\varepsilon)$ to the Dirac
operator on the hyperbolic plane $\HH$, true for any typical hyperbolic surface $X$ of high genus~$g$
with $o(\sqrt{g})$ cusps, and any nontrivial spin structure~$\varepsilon$ on $X$.

\begin{theorem}\label{eigenvaluesSharpEstimates}
  Let $(k(g))_{g \geq 2}$ be a sequence of non-negative integers such that
  $k(g) = o(\sqrt{g})$ as $g \rightarrow + \infty$. There exists a constant
  $C>0$ such that, for any $0\leq a \leq b$, any $g \geq 2$, any
  $X \in \mathcal{A}_{g,k(g)}$, and any nontrivial spin structure $\varepsilon$
  on $X$, we have
  \begin{align*}
    \frac{N^{\Da}_{(X,\varepsilon)} (a,b)}{\area (X)} 
    = 
    \frac{1}{4\pi} \int_a^b \eig \coth(\pi \eig) \d \eig + R(X,\varepsilon,a,b),
  \end{align*}
  where the remainder satisfies:
  \begin{align*}
    - C \frac{b+1}{\sqrt{\log g}} 
    \leq
    R(X,\varepsilon,a,b) 
    \leq 
    C \frac{b + 1}{\sqrt{\log g}}
    \left( 1+ \sqrt{\log \left( 2 + (b - a) \sqrt{\log (g)}\right)} \right).
  \end{align*}
\end{theorem}

A similar statement holds for the Laplacian spectrum, by work of the first author \cite{Monk} in the
compact case and Le Masson--Sahlsten \cite{LemassonSahlsten} when $k(g) \ll \sqrt{g}$. These results
have further been extended to twisted Laplacians by Gong \cite{Gong} very recently. The proof is
similar to the proof in \cite{Monk}, replacing the Selberg trace formula by a Dirac version from the
second author~\cite{Stan}.

Note that the support of the limiting measure is $[0, + \infty)$, because the spectrum of $\Da$ on
$\HH$ is $[0, + \infty)$. This contrasts with the (twisted) Laplacian setting, where the limiting
spectral density is
$\frac{1}{4 \pi} \tanh \left( \pi \sqrt{\eig - \frac 14} \right) \mathbbm{1}_{[\frac 14, +
  \infty)}(\eig) \d \eig$, supported on $[1/4, + \infty)$.

A remarkable aspect of Theorem \ref{eigenvaluesSharpEstimates} is that the limit we obtain is
independent of the nontrivial spin structure $\varepsilon$. The reason for that is that the
probabilistic assumption we make, and more precisely the Benjamini--Schramm hypothesis, makes the
geometric term of trace formulae subdominant, i.e. the spectra of $X$ converge to the spectra of
$\HH$, regardless of the precise geometry and spin structure on $X$.

\subsection{Upper bounds and pathological surfaces}
\label{sec:upper-bounds-path}

In the process of proving Theorem \ref{eigenvaluesSharpEstimates}, we prove the following upper
bound on the Dirac spectrum of typical hyperbolic surfaces. Throughout this article, when we write
$A = \mathcal{O}(B)$, we mean that there exists a constant $C>0$ such that, for any choice of
parameters, $|A| \leq C \, B$. We precise that the constant is allowed to depend on our choice of a
fixed sequence $(k(g))_{g \geq 2}$. If the constant depends on a parameter $p$, e.g. the genus $g$,
we rather write $A = \mathcal{O}_p(B)$.

\begin{proposition}\label{eigenvaluesBigOEstimates}
  With the notations of Theorem \ref{eigenvaluesSharpEstimates}, 
  \begin{align*}
    \frac{N^{|D|}_{(X,\varepsilon)} (a,b)}{\area (X)}
    = \mathcal{O} \left( (b+1)\left(b-a + \frac{1}{\sqrt{\log g}}
    \right)\right).
  \end{align*}
\end{proposition}

Building on results of B\"ar on pinched surfaces \cite{Bar}, we prove that such a bound cannot be
obtained for \emph{every spin hyperbolic surface}, because there exists ``pathological'' examples
for which the Dirac operator is discrete with arbitrarily many eigenvalues close to $0$.

\begin{proposition}
  \label{prop:pathological}
  Let $(g,k)$ be integers such that $2g-2+k>0$ and $g \geq 1$.  For any~$N$, any $\eta >0$, there
  exists a hyperbolic surface $X$ of signature $(g,k)$ and a nontrivial spin structure~$\varepsilon$
  on $X$ such that $N^{\Da}_{(X,\varepsilon)}(0,\eta) \geq N$.
\end{proposition}

This is another interesting difference between Laplacian and Dirac spectra. Indeed, the Laplacian
spectrum restricted to $[0,1/4]$ is discrete, and the number of eigenvalues under $1/4$ is at most
$2g-2+k$ by work of Otal--Rosas \cite{OtalRosas}. This is a \emph{topological bound}, in the sense
that it only depends on the topology of the hyperbolic surface. Proposition \ref{prop:pathological}
proves that such a bound cannot exist in the Dirac setting, while Proposition
\ref{eigenvaluesBigOEstimates} provides one true for any typical hyperbolic surface.

\subsection{Applications}

We deduce from Theorem \ref{eigenvaluesSharpEstimates} a uniform version of the Weyl law for Dirac
operators on typical hyperbolic surfaces (uniform in the sense that the rate of convergence is
independent of the surface $X \in \mathcal{A}_{g,k(g)}$ and the nontrivial spin structure
$\varepsilon$).

\begin{corollary}
  \label{coro:Weyl_law_typical}
  For any $g \geq 2$, any $X \in \mathcal{A}_{g,k(g)}$, any nontrivial spin structure $\varepsilon$
  on $X$, 
  \begin{equation*}
    \forall \eig \geq 1, \quad
    \frac{N^{\Da}_{(X,\varepsilon)} (0,\eig)}{\area (X)} 
    = 
    \frac{\eig^2}{8\pi} 
    + \mathcal{O}_g \left( \eig \sqrt{\log (\eig)} \right).
  \end{equation*}
\end{corollary}
The implied constant above only depends on the genus $g$, in a way that can be made explicit using
Theorem \ref{eigenvaluesSharpEstimates}.

Taking a shrinking interval of size $1/\sqrt{\log(g)}$ above $\eig$, we deduce from Proposition
\ref{eigenvaluesBigOEstimates} the following topological bound on the multiplicity
$\mathrm{mult}_{(X,\varepsilon)}(\eig)$ of any Dirac eigenvalue $\eig$.

\begin{corollary}
  \label{coro:multiplicity}
  For any $g \geq 2$, any $X \in \mathcal{A}_{g,k(g)}$, any $\eig \geq 0$,
  \begin{equation*}
    \frac{\mathrm{mult}_{(X,\varepsilon)}(\eig)}{\area(X)}
    = \mathcal{O} \left(\frac{\eig+1}{\sqrt{\log(g)}} \right).
  \end{equation*}
\end{corollary}

\subsection{Acknowledgements}
\label{sec:acknowledgements}

The authors would like to thank Sergiu Moroianu for valuable discussion and comments. This research
was funded by the EPSRC grant EP/W007010/1 ``Spectral statistics for random hyperbolic
surfaces''. The second author was partially supported from the project PN-III-P4-ID-PCE-2020-0794,
financed by UEFISCDI.

\section{Preliminaries}

\subsection{Spin structures and the Dirac operator}
\label{sec:Dirac}

In this subsection we briefly describe spin structures and introduce the Dirac
operator. For more details, we refer the reader to~\cite{spinorsBook1} and
\cite{spinorsBook2}.

Let $n$ be an even integer, and let us denote by $\cl_n$ the Clifford algebra associated to $\Rn$
with the standard scalar product. The subgroup $\spin(n) \subset \cl_n$ consists of all even
products of unit vectors. It can be shown that $\spin(n)$ is a connected, two-sheeted covering of
$\so(n)$, the group of $n\times n$ matrices of determinant $1$. Consider $\{e_1,...,e_{n} \}$ the
standard basis in~$\Rn$, and denote by $J$ the standard almost complex structure. The
representation:
\begin{align*}
  cl (v) : = \frac{1}{\sqrt{2}}(v-iJv)\wedge (\cdot) - \frac{1}{\sqrt{2}}(v+iJv)\lrcorner (\cdot),
\end{align*}
where $\lrcorner$ is obtained form the $\CC$-bilinear extension of the standard
scalar product, acts on $\Sigma_n := \wedge^*W$, where $W \subset \CC^n$ is
generated by
$\{\frac{1}{\sqrt{2}}(e_1-ie_2),...,\frac{1}{\sqrt{2}}(e_{n-1}-ie_{n})
\}$. One can check that this representation extends to the complexified Clifford
algebra
\[
cl:\cl_n\otimes_{\RR}\CC \longrightarrow \End_{\CC} (\Sigma_n).
\]

Consider $X$ a $n$-dimensional oriented manifold with a Riemannian metric. A
\emph{spin structure} on $X$ is a principal $\spin(n)$ bundle $P_{\spin (n)}X$
covering $P_{\so (n)}X$, the principal bundle of oriented orthonormal frames,
with two sheets. Moreover, this covering must be compatible with the group
covering $\spin(n) \longrightarrow \so(n)$. Once we have a fixed spin structure,
we can define the \emph{spinor} bundle as the associated vector bundle
$S:=P_{\spin(n)} X\times_{cl}\Sigma_n$. The Dirac operator is defined as
follows:
\begin{align*}
\D:C^{\infty}(S) \longrightarrow C^{\infty} (S), && \D:= cl \circ \nabla,
\end{align*}
where $\nabla$ is the connection on $S$ induced by the Levi-Civita connection on
$X$. One can see that $\D$ is an elliptic, self-adjoint differential
operator of order $1$.

It is known that orientable, complete hyperbolic surfaces of finite area always admit spin
structures. From now on, we restrict our attention to this type of surfaces. Let
$X = \G\setminus \HH$, where $\G$ is a subgroup of $\psl$ without elliptic elements for which the
associated surface is of finite area. We denote as $\pi$ the standard projection
$\pi : \slinear_2(\RR)\longrightarrow \psl$. We can reinterpret a spin structure on $X$ as a left
splitting in the following short exact sequence:
\begin{align*}
  1 \longrightarrow \{ \pm 1 \} \longrightarrow \tilde{\G}
  := \pi^{-1}(\G) \longrightarrow \G \longrightarrow 1,
\end{align*}
i.e. a morphism $\chi:\tilde{\G}\longrightarrow \{ \pm 1\}$ for which
$\chi \circ \iota = \id_{\{ \pm 1 \}}$, where $\iota$ is the natural inclusion.

We define $\varepsilon:\G \longrightarrow \{ \pm 1 \}$ by setting
$\varepsilon(\g) := \chi(\tilde{\g})$, where $\tilde{\g} \in \slinear_2(\RR)$ is the unique lift
with positive trace of $\g \in \psl$. This function is a class function, i.e. constant along
conjugacy classes. More details about the class function associated with a spin structure can be
found in \cite[Section 2]{Stan}. Further detail is also provided on the identifications between the
frame bundle $P_{\so (2)}\HH$ and the group of isometries $\psl$, and between the spin bundle
$P_{\spin(2)}\HH$ and the group $\slinear_2(\RR)$. If we consider $\rho$, a right splitting in the
above short exact sequence (which is uniquely determined by the left splitting $\chi$), we can
define the action of a $\g \in \G$ on $P_{\spin(2)}\HH$ by left multiplication with $\rho(\g)$. It
can be easily seen that this action descends to the spinor bundle $S$.

\subsection{The spectrum of Dirac operators}

The aim of this article is to study the spectrum of the Dirac operator acting on
a typical hyperbolic surface $X$ of finite area. 

\subsubsection{Cusps and nontrivial spin structures}
\label{sec:cusps-non-trivial}

The spectrum of the Dirac operator is always discrete when $X$ is
compact. Remarkably, when $X$ admits some cusps, the spectrum can either be
discrete or the real line $\RR$, depending on the spin structure.  More
precisely, B\"ar showed in \cite[Theorem 1]{Bar} that the spectrum of the Dirac
operator is discrete if and only if the spin structure is nontrivial along each
cusp of $X$. It is shown in \cite[Lemma 8]{Stan} that this is equivalent to
assuming that $\varepsilon(\g) = -1$ for any primitive parabolic element
$\g \in \G$.

Note that B\"ar further proved in \cite[Corollary 2]{Bar} that any finite area
hyperbolic surface admits at least one spin structure such that the spectrum is
discrete. 

\subsubsection{Multiplicities and counting functions}
\label{sec:spectr-mult}

As explained in \cite[Section 4]{Bolte}, the spectrum of the Dirac operator admits two rigid sources
of multiplicity: the \emph{chiral symmetry} and the \emph{time-reversal symmetry}.  It follows that
the spectrum of $\D$ is symmetric about $0$ (i.e. if $\eig \in \RR$ is an eigenvalue then
$-\eig$ is an eigenvalue), and every eigenvalue has even multiplicity. We conclude that the
multiplicity of every eigenvalue of $\Da$ is a multiple of $4$.

In order to avoid counting every eigenvalue exactly four times, we shall study the \emph{reduced
  spectrum} $(\eig_j)_{j \geq 0}$, where we let $\eig_j := \Lambda_{4j}$ for
$(\Lambda_j)_{j \geq 0}$ the ordered spectrum of $\Da$ (with multiplicities).  We then define, for
$0 \leq a \leq b$, the counting function
\begin{equation*}
  N^{\Da}_{(X,\varepsilon)}(a,b)
  := \{j \geq 0 \, : \, a \leq \eig_j \leq b\}
  = \frac{1}{4} \, \# \{ \text{eigenvalues of } \Da \text{ in } [a,b] \}.
\end{equation*}

\subsubsection{Pathological examples of Dirac spectra}
\label{sec:absence-topol-bound}

Let us now prove Proposition \ref{prop:pathological}, which claims that there is no topological
bound on the number of eigenvalues in $[0, \eta]$, provided $g >0$.  The proof relies on the two
following results, proven by B\"ar in \cite{Bar}.

\begin{lemma}
  \label{lem:switch_spin}
  Let $X$ be a finite area hyperbolic surface. For any simple non-separating closed geodesic
  $\gamma$ on $X$, there exists two nontrivial spin structures $\varepsilon_\pm$ on $X$ such that
  $\varepsilon_\pm(\gamma) = \pm 1$.
\end{lemma}

We precise that, in the previous statement, the geodesic $\gamma$ is \emph{simple} if it has no
self-intersection, and \emph{non-separating} if the surface $X \setminus \gamma$ obtained by cutting
$X$ along $\gamma$ is connected.  In other words, this lemma tells us that, if $\gamma$ is
non-separating, then we can pick the value of a nontrivial spin structure at $\gamma$ freely.

\begin{proof}
  This is a direct consequence of the discussion in \cite[page 481]{Bar}.
\end{proof}

\begin{lemma}
  \label{lem:pinch}
  Let $X$ be a finite area hyperbolic surface equipped with a nontrivial spin structure
  $\varepsilon$.  Let $\gamma$ be a simple non-separating geodesic on $X$ such that
  $\varepsilon(\gamma) = +1$.  Let $(X_n)_{n \geq 1}$ be a sequence of finite area hyperbolic
  surfaces obtained from $X$ by pinching the geodesic $\gamma$ so that its length goes to $0$ as
  $n \rightarrow \infty$.  Then, for any $\eta > 0$,
  \begin{equation*}
    N^{\Da}_{(X_n,\varepsilon)}(0,\eta)
    = - \frac{\eta}{\pi} \log (\ell_{X_n}(\gamma)) + \mathcal{O}_{\eta}(1).
  \end{equation*}
\end{lemma}

The pinching procedure mentioned above is a classic way to construct pathological examples in
hyperbolic geometry, and described in more detail in \cite[Section 1]{Bar}. Lemma \ref{lem:pinch}
quantifies how the Dirac spectrum of $X_n$ converges to the Dirac spectrum of the limit $X_\infty$
of $(X_n)_n$ as $n \rightarrow \infty$. There is an accumulation process because, on top of the
nontrivial cusps of $X$, $X_\infty$ has two new cusps with a trivial spin structure (coming from
the pinched geodesic), and hence the Dirac spectrum of $(X_\infty,\varepsilon)$ is $\RR$.

\begin{proof}
  This is a trivial adaptation of \cite[Theorem 2]{Bar} when $X$ is a surface of finite area
  equipped with a nontrivial spin structure, rather than a closed surface. No changes are required
  in the proof.
\end{proof}

We are now ready to prove Proposition \ref{prop:pathological}.

\begin{proof}
  Let $X$ be an arbitrary hyperbolic surface of signature $(g,k)$.  The genus of $X$ is nonzero,
  and hence there exists a simple non-separating geodesic $\gamma$ on $X$. By
  Lemma~\ref{lem:switch_spin}, since $\gamma$ is non-separating, there exists a nontrivial spin
  structure $\varepsilon$ on $X$ such that $\varepsilon(\gamma) = +1$. Then, we define a sequence of
  metrics $(X_n)_{n \geq 1}$ as in Lemma~\ref{lem:pinch} by pinching the geodesic~$\gamma$ so that
  its length goes to $0$ as $n \rightarrow + \infty$. By Lemma~\ref{lem:switch_spin}, since
  $\varepsilon(\gamma)=+1$, 
  \begin{equation*}
    N^{\Da}_{(X_n,\varepsilon)}(0,\eta)
    = - \frac{\eta}{\pi} \log (\ell_{X_n}(\gamma)) + \mathcal{O}_{\eta}(1)
    \underset{n \rightarrow + \infty}{\longrightarrow} + \infty
  \end{equation*}
  for any fixed $\eta >0$. In particular we can pick a $n$ such that
  $N^{\Da}_{(X_n,\varepsilon)}(0,\eta) \geq N$. Then, $X_n$ satisfies our claim.
\end{proof}

\subsubsection{The Selberg trace formula for Dirac operators}
\label{sec:selb-trace-form}

Our main tool to study the counting function $N^{\Da}_{(X,\varepsilon)}(a,b)$
is the Selberg trace formula for the Dirac operator on compact hyperbolic
surfaces, developed by Bolte and Stiepan in~\cite{Bolte} and generalised to
hyperbolic surfaces of finite area by the second author \cite[Theorem
13]{Stan}.  This formula relates the Dirac spectrum of a finite area hyperbolic
surface $X$ to its length spectrum, i.e. the list of the lengths of all closed
geodesics on $X$, under the condition that the spin structure is nontrivial.
In this article, following the line of \cite{Monk}, we will use the
following pretrace formula, adapted from \cite[formula (10)]{Stan}.

\begin{theorem}\label{preTraceSelbergDirac}
  Consider $X=\G \backslash \HH$ a hyperbolic surface with $k$ cusps,
  equipped with a nontrivial spin structure $\varepsilon$. Let
  $(\eig_j)_{j\in {\mathbb{N}}}$ denote the reduced spectrum of
  $\Da$.
  Then, for any admissible test function $h$,
  \begin{equation}
    \label{eq:pretrace}
    \begin{split}
      \sum_{j=0}^{\infty} h(\eig_j) = & \, \frac{\area(X)}{8\pi}\int_{\RR} 
      h(\eig)\, \eig \coth(\pi \eig)\d \eig - \frac{ \log(2)}{2} \, k  \, \check{h}(0) \\
      & + \frac 12
      \sum_{\g\neq 1}\int_\fund \varepsilon(\g) \, K(d(z,\g z)) \, \tau_{z\mapsto \g^{-1}z} \d z
    \end{split}
  \end{equation}
  where  the sum is taken after all hyperbolic elements in $\G$ and:
  \begin{itemize}
  \item the set $\fund$ is a fundamental domain of $X = \G \backslash \HH$;
  \item $\tau_{z\mapsto w} = - i \frac{z-\overline{w}}{|z-\overline{w}|}$
    is the parallel transport of spinors from $z$ to $w$ with respect to $\nabla$;
  \item the kernel $K$ can be expressed as:
    \begin{equation}
      \label{kernelFormula}
      K(r)
      = -\frac{\cosh \left( \tfrac{r}{2}
      \right)}{\pi\sqrt{2}}\int_r^{\infty}\frac{\check{h}'(\rho)-
      \tfrac 12 \check{h}(\rho)
      \tanh \left( \tfrac{\rho}{2}\right)}
      {\cosh \left( \tfrac{\rho}{2}\right)
      \sqrt{\cosh (\rho) - \cosh (r)}}\d \rho
  \end{equation}
  where $\check{h}$ is the inverse Fourier transform of $h$, i.e.
  \begin{align*}
    \check{h}(\rho) = \frac{1}{2\pi}\int_{\RR}h(\eig)e^{i\eig \rho}\d \eig.
\end{align*}
\end{itemize}
\end{theorem}

In the above statement, by admissible, we mean that there exists
$\eta > 0$ such that $h$ is an even holomorphic function defined
on the strip $\{z = x+iy \, : \, |y| \leq \tfrac 12 + \eta\}$
which satisfies $|h(z)| \leq C (|z|^2+1)^{-1-\eta}$ for a
constant $C>0$, as in \cite{Bolte,Monk}.

\begin{proof}
  In \cite[formula (10)]{Stan}, the second author proved that,
  under the hypotheses of the theorem, for any
  $\phi\in C^{\infty}_c(\RR)$, if we set
  \begin{equation}
    \label{eq:trace_formula_expr_wh}
    w(x):=\frac{\phi(x)}{\sqrt{x+4}}, 
    \quad \text{and} \quad
    \check{h}(x):=4\cosh\left(\tfrac{x}{2} \right)
    \int_{0}^{\infty}
    w\left( 4 \sinh^2 \left(\tfrac{x}{2} \right)+y^2 \right)\d y,
  \end{equation}
  then \eqref{eq:pretrace} holds with the kernel
  $K(r) := \phi(4 \sinh^2 (r/2))$. As a consequence, in order to
  conclude, all that we have to do is to associate a function $\phi$
  to our test function $h$, and hence express the kernel $K$ in
  terms of $h$.  To do so, we shall consider the following operators
  $\A$, $\B$ acting on the set $\mathcal{S}([0,\infty])$ of Schwartz
  functions on $[0,\infty]$:
\begin{align*}
  \A\varphi (x) :=\int_0^{\infty}\varphi(x+y^2) \d y,
  &&
     \B\psi (x) :=-\frac{4}{\pi}\int_0^{\infty}\psi '(x+y^2) \d y.
\end{align*}
By direct computations, one can easily see that $\B \circ
\A=1$. Indeed:
\begin{align*}
  (\B \circ \A)\psi (x)
  &=-\frac{4}{\pi} \int_0^{\infty}\int_0^{\infty}\psi'(x+y^2+z^2)
    \d z \d y
    = \frac{-4}{\pi}
    \int_{[0,\tfrac{\pi}{2}]\times[0,\infty)}\psi'(x+r^2)r \d \theta
    \d r\\
  &=-2\int_0^{\infty}\frac{1}{2}\frac{\partial \psi(x+r^2)}
    {\partial r} \d r = \psi(x).
\end{align*}
Writing the identity in such a way allows us to compute the kernel
in terms of $\check{h}$.  On the one hand, by the expression of
$\check{h}$ in equation \eqref{eq:trace_formula_expr_wh}, we get that:
\begin{align*}
  \A w \circ \left(4 \sinh^2\left( \tfrac{\cdot}{2} \right)\right)
  =   \frac{\check{h}}{4\cosh \left( \tfrac{\cdot}{2}\right)}
\end{align*}
thus, it follows that:
\begin{equation}
  \label{eq:expr_Aw}
  4\sinh \left( \tfrac{\rho}{2}\right)
  \cosh \left( \tfrac{\rho}{2}\right)
  (\A w)' \left(4 \sinh^2\left( \tfrac{\rho}{2} \right)\right)
  =
  \frac{\check{h}'(\rho) \cosh\left( \frac{\rho}{2}
    \right)-\frac{1}{2}\check{h}(\rho) \sinh\left( \frac{\rho}{2}
    \right)}
  {4\cosh^2 \left( \frac{\rho}{2}\right)}.
\end{equation}
On the other hand, by definition of $K$ and $w$, since $\sqrt{4
  \sinh^2(\tfrac r 2)+4}=2 \cosh (\tfrac r 2)$,
\begin{align*}
  K(r) & = \phi\left(4 \sinh^2\left( \tfrac r2 \right)\right)
  = 2 \cosh \left(\tfrac{r}{2} \right) 
  w\left(4 \sinh^2\left( \tfrac r2 \right)\right) \\
  & = - \frac{8}{\pi} \cosh \left(\tfrac{r}{2} \right)
  \int_0^{\infty}(\A w)' \left(
    4 \sinh^2\left( \tfrac r2 \right)+y^2 \right)\d y
\end{align*}
because $\B \circ \A = 1$. We then perform the change of variable
$4 \sinh^2 (\tfrac \rho 2) = 4 \sinh^2(\tfrac r 2)+y^2$ and obtain
the claimed expression thanks to \eqref{eq:expr_Aw} and the fact
that
\begin{equation*}
  y = \sqrt{4\sinh^2 (\tfrac \rho 2) - 4\sinh^2 (\tfrac r 2)}
  = \sqrt{2 (\cosh(\rho) - \cosh(r))}.
\end{equation*}
\end{proof}

\subsection{Random hyperbolic surfaces}
\label{s:random_hyp}

In this article, we study the properties of random hyperbolic surfaces sampled with the
Weil--Petersson probability measure. Let us provide the key elements that are necessary for the
reading of this article -- thorough presentations of this probabilistic model are provided in
\cite{Monk_thesis,Wright}.

Let $g, k$ be integers such that $2g-2+k>0$. Our sample space is the
\emph{moduli space}
\begin{equation*}
  \mathcal{M}_{g,k} := \{ \text{hyperbolic surfaces of signature } (g,k) \}
  \diagup \text{isometries}. 
\end{equation*}
This space is an orbifold of dimension $6g-6+2k$. Weil introduced in \cite{Weil} a natural
symplectic structure on $\mathcal{M}_{g,k}$, called the \emph{Weil--Petersson form}. It induces a
volume form of finite volume, which can be renormalised to obtain a probability measure
$\mathbb{P}_{g,k}$ on the moduli space $\mathcal{M}_{g,k}$.

Our objective is to describe ``typical behaviour'', i.e. we will focus on
proving properties true with probability going to one in a certain asymptotic
regime. More precisely, for our fixed sequence $(k(g))_{g \geq 2}$, we will say a
property is true \emph{with high probability} in the large genus limit if
\begin{equation*}
  \lim_{g \rightarrow + \infty} \mathbb{P}_{g,k(g)}(X \in \mathcal{M}_{g,k(g)}
  \text{ satisfies the property}) = 1.
\end{equation*}
The following result states two key geometric properties true with high
probability which we will use in this article.

\begin{theorem}\label{probabilityOfAg}
  Let $(k(g))_{g \geq 2}$ be a sequence of non-negative integers such that
  $k(g) = o(\sqrt{g})$ as $g \rightarrow + \infty$. Then, for all $g\geq 2$,
  there exists a subset $\mathcal{A}_{g,k(g)}$ of the moduli space
  $\mathcal{M}_{g,k(g)}$ of probability $1 - \mathcal{O}(\log(g) g^{-1/12})$ such that
  any surface $X \in \mathcal{A}_{g, k(g)}$ satisfies the following.
  \begin{itemize}
  \item If $X^- (L)$ is the $L$-thin part of $X$, i.e. the set of points in $X$
    with radius of injectivity shorter than $L$, then
    \begin{align*}
      \frac{\area \left( X^- ( \frac{1}{6}\log g) \right)}{\area (X)} \leq g^{-\frac{1}{3}}.
    \end{align*}
  \item The systole of $X$ (i.e. its shortest closed geodesic) is longer than
    $g^{-\frac{1}{24}} \sqrt{\log g}$.
  \end{itemize}
\end{theorem}

\begin{proof}
  The first point was proven by the first author in \cite[Corollary
  4.4]{Monk_thesis}, and the second by Le Masson and Sahlsten in \cite[Lemma A.1]{LemassonSahlsten}.
\end{proof}

\section{Plan of the proof and first estimates}

In this section, we set up some notations in order to prove
Theorem~\ref{eigenvaluesSharpEstimates}, following the lines of \cite{Monk},
and prove first easy estimates.

\subsection{The family of test functions}
 
A key step of the proof is to construct a family of test functions such that the
spectral side of the Selberg trace formula is a good approximation of the
counting number $N^{\Da}_{(X,\varepsilon)} (a,b)$, for $0 \leq a \leq b$. Our choice of test
function is a straightforward adaptation to the choice made in \cite[Section
4]{Monk}.

For $t>0$, a parameter which will grow like $\sqrt{\log g}$, consider the family
of test functions $h_t:\CC \longrightarrow \CC$ defined by the convolution
\begin{align*}
  h_t(\eig)
  := (\mathbbm{1}_{[a,b]}\star v_t)(\eig)
  = \frac{t}{\sqrt{\pi}} \int_{a}^{b} \exp\left(-t^2(\eig-\rho)^2\right)\d \rho,
\end{align*}
where $\mathbbm{1}_{[a,b]}$ is the indicator of the segment $[a, b]$
and $v_t(x):=\frac{t}{\sqrt{\pi}}\exp\left(-t^2x^2\right)$ is the Gaussian of
mean $0$ and variance $\frac{1}{t}$. One can easily see that $h_t$ is
holomorphic. Since it is not even, we will rather apply Proposition
\ref{preTraceSelbergDirac} to the function $H_t(\eig) := h_t(\eig)+h_t(-\eig)$.

Let us present elementary properties of the functions $h_t$ and $g_t := \check{H}_t$ proven in
\cite{Monk}, which will be useful to the proof of Theorem \ref{eigenvaluesSharpEstimates}.

\begin{lemma}
  \label{lem:bounds_test_function}
  Let $0 \leq a \leq b$ and $t > 0$.
  \begin{enumerate}
  \item The function $H_t$ is admissible.
  \item \label{inversFourierBound}
    For any $t>0$, and for any $x>0$ we have:
    \begin{align*}
      |g_t(u)|
      & \leq
        \frac{2}{\pi u}\exp\left( -\frac{u^2}{4t^2} \right) \\
      |g_t'(u)|
      & \leq \left( \frac{1}{\pi t^2}
        + \frac{4b}{\pi u} \right)
        \exp\left( -\frac{u^2}{4t^2}  \right).
    \end{align*}
  \item \label{htBound} As $t \rightarrow + \infty$, $h_t$ converges to the function
    $\eig \mapsto \tilde{\mathbbm{1}}_{[a,b]}(\eig)$ which coincides with $\mathbbm{1}_{[a,b]}$
    except at $\eig = a$ and $b$ where it is equal to $1/2$. More precisely, for $\eig \in \RR$,
    \begin{align*}
      |h_t(\eig)-\tilde{\mathbbm{1}}_{[a,b]}(\eig)|\leq
      \begin{cases} 
        s(t|\eig-a|)
        & \text{if }
          \eig\in (-\infty, a) \cup \{ b \} \\
        s(t|\eig-a|)+s(t|\eig-b|)
        & \text{if } \eig\in (a,b) \\
        s(t|\eig-b|)
        & \text{if } \eig\in \{ a \} \cup ( b, \infty) 
      \end{cases}
    \end{align*}
    where $s:(0,\infty)\longrightarrow \RR$ is the non-increasing
    function $s(\rho):=\frac{\exp(-\rho^2)}{2\sqrt{\pi}\rho}$.
  \end{enumerate}
\end{lemma}

\begin{proof}
  These three points are respectively Lemma 10, 13 and 21 from
  \cite{Monk}.
\end{proof}

\subsection{Plan of the proof}
\label{sec:sketch-proof}

In order to prove our main result, we apply Theorem \ref{preTraceSelbergDirac} to the family of
functions $H_t$, of kernels $K_t$, and obtain that for any hyperbolic surface $X$ of signature
$(g,k)$,

\begin{align}\label{traceSumFarFrom0}
\begin{split}
\frac{1}{\area (X)} \sum_{j=0}^{\infty} H_t(\eig_j) 
&= 
\frac{1}{8\pi}\int_{\RR}H_t(\eig) \eig \coth(\pi \eig)\d \eig
-
\frac{k \log(2)}{2\area (X)} \, g_t(0) \\
&+
\frac{1}{2\area (X)}\sum_{\g \neq 1}\int_\fund \varepsilon(\g) K_t(z,\g z) \tau_{z\mapsto \g^{-1}z}\d z. 
\end{split}
\end{align}
The left hand side of this formula is an approximation of the ratio
$N^{\Da}_{(X,\varepsilon)}(a,b)/\area(X)$, which we wish to estimate. Thus, we
shall study the right hand side, term by term.
\begin{itemize}
\item In Section \ref{s:integral}, we bound the difference between the integral term
  \begin{equation*}
    I(t,a,b) :=\frac{1}{8\pi}\int_{\RR}H_t(\eig) \, \eig \coth(\pi \eig)\d \eig
  \end{equation*}
  and the integral that appears in Theorem \ref{eigenvaluesSharpEstimates}.
\item Then, in Section \ref{s:cusp}, we prove an easy bound on the cuspidal term
  \begin{equation*}
    C(X,t) := - \frac{k \log(2)}{2\area (X)} \, g_t(0).
  \end{equation*}
\item Section \ref{s:kernel} is dedicated to bounding the kernel term,
  \begin{equation*}
    R_K(X,\varepsilon,t,a,b)
    :=\frac{1}{2\area (X)}\sum_{\g \neq 1}\int_\fund \varepsilon(\g) \, K_t(z,\g z) \tau_{z\mapsto \g^{-1}z}\d z
  \end{equation*}
  which is the most difficult part of the analysis of the trace formula, where the probabilistic
  assumption on $X \in \mathcal{A}_{g,k(g)}$ is necessary.
\end{itemize}
We then conclude to the proof of Theorem
\ref{eigenvaluesSharpEstimates} in Section \ref{s:end_proof}, where
we compare the left hand side of \eqref{traceSumFarFrom0} with the
rescaled number of $\Da$-eigenvalues between $a$ and $b$.

\subsection{Asymptotic of the integral term}
\label{s:integral}

Let us prove the following result, which bounds the difference between the
integral $I(t,a,b)$ and integral appearing in our claim.

\begin{proposition} \label{integralEstimate}
For any $t>0$, 
\begin{align*}
\left| I(t,a,b) - \frac{1}{4\pi}\int_a^b \eig \coth(\pi \eig)\d \eig \right| 
\leq \frac{b + 1}{2t} + \frac{1}{2t^2}. 
\end{align*}
\end{proposition}

\begin{proof}
We start by rewriting $I(t,a,b)$ in a more convenient form. First we use the
parity of $H_t$, then we write $H_t(\eig) = h_t(\eig)+h_t(-\eig)$, to obtain
\begin{align*}
\begin{split}
I(t,a,b)
&=
\frac{1}{8\pi}\int_{\RR} H_t(\eig) \, \eig \coth(\pi \eig)\d \eig
= \frac{1}{4\pi}\int_{\RR} h_t(\eig) \, \eig\coth(\pi \eig)\d \eig \\ 
&= \frac{1}{4\pi}\int_{\RR}
\left( h_t(\eig) - \tilde{\mathbbm{1}}_{[a,b]}(\eig) \right) \eig\coth(\pi \eig)\d \eig
+\frac{1}{4\pi}\int_{a}^{b}\eig \coth(\pi \eig)\d \eig.
\end{split}
\end{align*}
We shall use Lemma \ref{lem:bounds_test_function}.(\ref{htBound}) to bound the difference
$|h_t(\eig) - \tilde{\mathbbm{1}}_{[a,b]}(\eig)|$ appearing in the equation above.  Since the
function $s$ from this bound has a pole at $0$, we shall use different estimates near $\eig=a$ and
$\eig=b$. Thus we write the real line as a union of (up to) five intervals:
\begin{align*}
\RR = 
\left( -\infty,a-\frac{1}{t} \right] 
\cup \left[ a-\frac{1}{t},a+\frac{1}{t} \right] 
\cup \left[ a+\frac{1}{t},b-\frac{1}{t} \right] 
\cup \left[ b-\frac{1}{t},b+\frac{1}{t} \right]
\cup \left[ b+\frac{1}{t},\infty \right). 
\end{align*}
If $a + \frac{1}{t} > b -\frac{1}{t}$ then the interval in the middle
is omitted from the union. The integral splits accordingly into five parts,
denoted $I_{j}$, for $1\leq j \leq 5$:
\begin{equation}
  \label{e:bound_I1to5}
  \begin{split}
    & \left|I(t,a,b) - \frac{1}{4\pi}\int_a^b \eig \coth(\pi \eig)\d
    \eig\right| \\
    & \leq \frac{1}{4\pi} 
      \int_{\RR}\left| h_t(\eig) -\tilde{\mathbbm{1}}_{[a,b]} (\eig)  \right| \eig\coth(\pi \eig)\d \eig
      := \frac{1}{4\pi}  \sum_{j=1}^5 I_j.
  \end{split}
\end{equation}

We start with $I_1, I_3$ and $I_5$. Throughout the computations, we will use the
fact that:
\begin{equation}
  \label{e:bound_coth}
  \forall \eig \in \RR, \quad 0 \leq \eig\coth(\pi \eig) \leq |\eig|+\frac{1}{\pi}.
\end{equation}
Equation \ref{e:bound_coth} and Lemma \ref{lem:bounds_test_function}.(\ref{htBound}) for $\eig < a$
together imply
\begin{align*}
  I_1
  & 
    \leq \int_{-\infty}^{a-\frac{1}{t}}
    \frac{\exp\left( -t^2(\eig-a)^2 \right)}{2\sqrt{\pi} \, t (a-\eig)}
    \left( |\eig|+\frac{1}{\pi} \right)\d \eig  
   = \int_{1}^{\infty} \frac{e^{-x^2}}{2\sqrt{\pi}x}
    \left( \left| a - \frac{x}{t}\right| +\frac{1}{\pi} \right)\frac{\d x}{t} 
\end{align*}
by the change of variable $x = t (a - \eig) \in (1, + \infty)$. Then,
\begin{equation*}
  I_1
  \leq \frac{1}{2\sqrt{\pi}}\int_{1}^{\infty} e^{-x^2}
  \left(  a +1 + \frac{x}{t}  \right)\frac{\d x}{t}
    \leq \frac{a + 1 }{4t} + \frac{1}{4 t^2}.
\end{equation*}
With similar approximations one can also prove that:
\begin{align*}
I_3
\leq
\frac{a+b +2}{4t} + \frac{1}{2t^2} 
\quad \text{and} \quad
I_5
\leq
\frac{b + 1}{4t} + \frac{1}{4t^2}.
\end{align*}

For $I_2$ and $I_4$ we rather use the loose bound
$\left| h_t(\eig) -\tilde{\mathbbm{1}}_{[a,b]} (\eig) \right| \leq 1$, which
yields
\begin{align*}
  I_2
  & \leq \int_{a -\frac{1}{t}}^{a +\frac{1}{t}} \left( |\eig|+\frac{1}{\pi}\right)\d \eig
    \leq \int_{a -\frac{1}{t}}^{a +\frac{1}{t}}
    \left( a +1  + \frac{1}{t} \right) \d \eig
    \leq \frac{2a +2}{t} + \frac{2}{t^2}.
\end{align*}
In the same way we also get $I_4 \leq \frac{2b +2}{t} + \frac{2}{t^2}$.

Combining those bounds and using $a \leq b$, we obtain
\begin{align*}
I_1+I_3+I_5 
\leq
\frac{b + 1}{t} + \frac{1}{t^2}
\quad \text{and} \quad
I_2+I_4
\leq
\frac{4b + 4}{t} + \frac{4}{t^2}
\end{align*}
which allows us to conclude using equation \eqref{e:bound_I1to5} and $5/(4\pi) < 1/2$.
\end{proof}

\subsection{Bond of the cusps contribution}
\label{s:cusp}

Let us now prove the following.

\begin{proposition}
  \label{cuspsEstimate}
  For any $X$ of signature $(g,k)$ and any $t > 0$,
  \begin{equation*}
    |C(X,t)| \leq  \frac{k}{2g-2+k} (b-a).
  \end{equation*}
\end{proposition}

\begin{remark}
  \label{rem:cuspsEstimate_applied}
  We note that, in Theorem \ref{eigenvaluesSharpEstimates}, we are placed in the
  regime $k = k(g) = o(\sqrt{g})$. Then, Proposition \ref{cuspsEstimate} implies
  \begin{equation}
    \label{e:cuspsEstimate_applied}
    C(X,t)
    = o \left( \frac{b}{\sqrt{g}} \right).
  \end{equation}
\end{remark}

\begin{proof}[Proof of Proposition \ref{cuspsEstimate}]
  By definition, the cusp term is
  \begin{equation*}
    C(X,t) = - \frac{k \log(2)}{2\area (X)} \, g_t(0)
    = - \frac{k \log(2)}{8\pi (2g-2+k)} \, g_t(0)
  \end{equation*}
  by the Gauss--Bonnet theorem. We therefore simply have to estimate
  $g_t(0) = \frac{1}{2 \pi} \int_\RR H_t(\eig) \d \eig$. We observe that
  \begin{equation*}
    \int_{\RR}h_t(\eig)\d \eig  =
      \frac{t}{\sqrt{\pi}} \int_{\RR}\int_{a}^{b} \exp\left(-t^2(\eig-\rho)^2 \right)\d \rho \d \eig \\
    =
      \frac{t}{\sqrt{\pi}}\int_{a}^{b} \int_{\RR} \exp\left(-t^2x^2 \right)\d x \d \rho \\
    =b - a. 
  \end{equation*}
  Similarly we have that $\int_{\RR}h_t(-\eig)\d \eig = b - a$. Hence,
  $g_t(0)= (b- a) / \pi$, which leads to the claim.
\end{proof}

\section{Bound of the kernel term}
\label{s:kernel}

In what follows we show a bound on the kernel term of the trace formula,
\begin{equation*}
  R_K(X,\varepsilon,t,a,b)
  :=\frac{1}{2\area (X)}\sum_{\g \neq 1}
  \int_\fund \varepsilon(\g) \, K_t(z,\g z) \, \tau_{z\mapsto \g^{-1}z}\d z
\end{equation*}
where the summation runs over hyperbolic elements in the group $\G$ which are
not the identity, and $\fund$ is a fundamental domain of $X = \G \backslash \HH$.

The steps of the kernel bound are as follows.
\begin{itemize}
\item First, in Section \ref{sec:kernel-estimate}, we prove an upper bound on
  the values $K_t$ of the kernel appearing in $R_K$.
\item We prove a classic counting bound on hyperbolic elements of $\G$ in
  Section \ref{sec:bound-numb-hyperb}, in order to deal with the summation.
\item We then cut the fundamental domain $\fund$ in a thick and thin part,
  $\fund^\pm(L)$, in Section~\ref{sec:thin-thick-decomp}. We bound the integrals over these two sets
  separately.
\item Finally, in Section \ref{sec:prob-kern-estim}, we conclude to a
  quantitative probabilistic bound on $R_K$, using the probabilistic assumption
  from Theorem \ref{probabilityOfAg}, and in particular the Benjamini--Schramm hypothesis.
\end{itemize}

\subsection{Kernel estimate}
\label{sec:kernel-estimate}

Let us prove the following bound on $K_t$.

\begin{proposition}\label{kernelEstimate}
  For any $\rho, t >0$ we have:
  \begin{equation*}
    |K_t(\rho)| \leq
    \left( \frac{1}{ t^2} + \frac{b + 1}{\rho} \right)
    \left(1  + \frac{t}{\rho}  \right)
  \exp \left(-\frac{\rho^2}{4t^2} \right).
  \end{equation*}
\end{proposition}

\begin{proof}
  We start from the kernel formula, equation \eqref{kernelFormula}, and make use of the inequalities
  on $g_t$ and $g_t'$ obtained in Lemma
  \ref{lem:bounds_test_function}.(\ref{inversFourierBound}). More precisely,
  \begin{align*}
    |K_t(\rho)|
    & = \frac{\cosh \left( \tfrac{\rho}{2}
      \right)}{2\pi\sqrt{2}}
      \left|\int_\rho^{\infty}\frac{2 g_t'(u)-
       g_t(u)
      \tanh \left( \tfrac{u}{2}\right)}
      {\cosh \left( \tfrac{u}{2}\right)
      \sqrt{\cosh (u) - \cosh (\rho)}}\d u \right| \\
    & \leq \frac{1}{2\pi\sqrt{2}}
      \int_\rho^{\infty}\frac{2 |g_t'(u)| + |g_t(u)|}
      {\sqrt{\cosh (u) - \cosh (\rho)}}\d u  
  \end{align*}
  by the triangle inequality. By Lemma \ref{lem:bounds_test_function}.(\ref{inversFourierBound}),
  \begin{equation*}
    2 |g_t'(u)| + |g_t(u)|
    \leq \frac{2}{\pi} \left( \frac{1}{t^2}  +\frac{4b+1}{u} \right)
    \exp \left(- \frac{u^2}{4t^2} \right)
  \end{equation*}
  and hence
  \begin{equation*}
    |K_t(\rho)|
    \leq
    \frac{1}{\pi^2 \sqrt{2}}
    \left( \frac{1}{ t^2} + \frac{4b + 1}{\rho} \right)
    \int_{\rho}^{\infty}
    \frac{ \exp \left(-\tfrac{u^2}{4t^2} \right) }{\sqrt{\cosh u - \cosh \rho}} \d u.
  \end{equation*}

  We then proceed with the splitting of the integral at $u=2\rho$. If
  $u\in[\rho, 2\rho]$, then
  $\cosh u - \cosh \rho \geq (u-\rho)\sinh \rho \geq (u-\rho)\rho$. Hence:
\begin{align*}
  \int_{\rho}^{2\rho} \frac{ \exp \left( -\frac{u^2}{4t^2} \right) }
  {\sqrt{\cosh u - \cosh \rho}} \d u 
   \leq 
    \exp \left(-\frac{\rho^2}{4t^2} \right)\int_{\rho}^{2\rho} \frac{\d u}{\sqrt{(u-\rho)\rho}}
  =  2  \exp \left(-\frac{\rho^2}{4t^2} \right).
\end{align*}
In the other case, if $u\in [2\rho, \infty)$, we can deduce that
$\cosh u - \cosh \rho \geq \frac{1}{2}(u^2-\rho ^2)\geq \frac{3}{2}\rho^2$. It
follows that:
\begin{align*}
  \int_{2\rho}^{\infty} 
  \frac{\exp \left( -\tfrac{u^2}{4t^2} \right) }{\sqrt{\cosh u - \cosh \rho}} \d u 
  & \leq 
    \frac{\sqrt{2}}{\sqrt{3} \, \rho} \int_{2\rho}^{\infty}
    \exp \left(-\frac{u^2}{4t^2} \right) \d u \\
  & \leq 
    \frac{\sqrt{2}}{\sqrt{3} \, \rho} \int_{0}^{\infty}
    \exp \left(-\frac{u^2}{4t^2} - \frac{\rho^2}{t^2} \right) \d u 
    = \frac{\sqrt{2 \pi} \, t}{\sqrt{3}\rho} 
    \exp \left(- \frac{\rho^2}{t^2} \right).
\end{align*}
Finally, putting everything together, we obtain:
\begin{equation*}
  |K_t(\rho)| \leq
  \frac{1}{\pi^2 \sqrt{2}}
  \left( \frac{1}{ t^2} + \frac{4b + 1}{\rho} \right)
  \left( 2  + \frac{\sqrt{2 \pi} \, t}{\sqrt{3}\rho}  \right)
  \exp \left(-\frac{\rho^2}{4t^2} \right)
\end{equation*}
which implies our claim.
\end{proof}

\subsection{Bound on the number of hyperbolic elements}
\label{sec:bound-numb-hyperb}

We shall use the following classic bound in order to control the summations over
hyperbolic elements of $\G$.

\begin{lemma}\label{geodesicNumberBound}
  Let $r \leq 2$ be a positive number and let $X = \G \backslash \HH$
  be a hyperbolic surface whose systole is larger than $2r$. Then,
  for any $j>0$, any $z \in \HH$,
  \begin{align*}
    \# \{ \g \in \G \setminus \{1\} : \g \text{ hyperbolic, }
    d(z,\g z) \leq j \} \leq \frac{4e^{1+j} }{r^2}.
  \end{align*}
\end{lemma}

\begin{proof}
  Choose $z\in \HH$ a point. The family of disks centred at $\g z$
  and of radius $r/2$, for hyperbolic elements $\g \in \G$, are
  disjoint. By comparing areas, the number of hyperbolic elements
  $\g$ for which $d(z,\g z)\leq j$ must be smaller than:
  \begin{equation*}
    \frac{\cosh\left( j+ \frac{r}{2}
      \right)-1}{\cosh\left(\frac{r}{2} \right)-1} \leq
    \frac{e^{j+\frac{r}{2}}}{\frac{r^2}{4}} \leq \frac{4e^{1+j}
    }{r^2}. 
  \end{equation*}
\end{proof}

\subsection{Thin-thick decomposition of the fundamental domain}
\label{sec:thin-thick-decomp}

For a positive real number $L$, we decompose the fundamental domain $\fund$ of
$X = \G \backslash \HH$ as a disjoint union of two sets $\fund^-(L)$ and $\fund^+(L)$, the points of
$\fund$ with injectivity radius smaller than $L$ and larger than $L$ respectively. Splitting the
integral in the sum $R_K$ into two integrals, on those two sets, we can rewrite our sum as:
\begin{align*}
R_K(X,\varepsilon,t,a,b)=R^-_K(X,\varepsilon,t,a,b,L)+R^+_K(X,\varepsilon,t,a,b,L).
\end{align*}

We shall start by bounding the contribution given by integration on $\fund^+(L)$,
using the fact that all points on $\fund^+(L)$ have an injectivity radius larger
than $L$.

\begin{lemma}\label{rPlusEstimate}
  Let $t > 0$ and $0 < r \leq 2$. Suppose that $X=\G \setminus \HH$ is a
  hyperbolic surface whose systole is larger than $2r$. If $L$ is a real number
  such that $L \geq 8t^2$, then:
\begin{align*}
  |R^+_K(X,\varepsilon,t,a,b,L)| 
  \leq
  \frac{4e}{r^2} \left( \frac{1}{ t^2} + \frac{b + 1}{r} \right)
    \left(1  + \frac{t}{r}  \right)
  \exp \left( - L  \right).
\end{align*}
\end{lemma}

\begin{proof}
  By definition of $\fund^+(L)$, for $z\in \fund^+(L)$, the sum defining
  $R^+_K$ contains no elements~$\g$ such that $d(z,\g z) < L$. Thus,
  we can write:
  \begin{align*}
    R^+_K(X,\varepsilon,t,a,b,L)
    =
    \frac{1}{2\area(X)}
    \int_{\fund^+(L)}\sum_{j= \lfloor L \rfloor}^{\infty}
    \sum_{\substack{\g\neq 1 \\ j\leq d(z,\g z)<j+1}} \varepsilon(\g) \, K_t(z,\g z) \,  
    \tau_{z\mapsto \g^{-1}z}\d z.
  \end{align*}
  When bounding the quantity above by the triangular inequality, we note that
  $|\varepsilon(\g)|=|\tau_{z\mapsto \g^{-1}z}|=1$, which allows to
  safely ignore these terms. Moreover, notice that the distance between $z$ and
  $\g z$ is always larger than $r$, the injectivity radius. Thus,
  Proposition~\ref{kernelEstimate} together with Lemma \ref{geodesicNumberBound}
  imply:
  \begin{align*}
    |R^+_K(X,\varepsilon,t,a,b,L)| 
    \leq \frac{2e}{r^2} \left( \frac{1}{ t^2} + \frac{b + 1}{r} \right)
    \left(1  + \frac{t}{r}  \right)
      \frac{\area(\fund^+(L))}{\area (X)}
      \sum_{j= \lfloor L \rfloor}^{\infty}
      \exp\left( j - \frac{j^2}{4t^2} \right).
  \end{align*}
  First, we note that $\area(\fund^+(L)) \leq \area(X)$. We then observe that,
  provided $L\geq 8t^2$, we have $j\leq \frac{j^2}{8t^2}$ and hence by
  comparison of the sum with an integral,
  \begin{align*}
    \sum_{j= \lfloor L \rfloor}^{\infty} \exp\left( j - \frac{j^2}{4t^2} \right)
    & \leq
       \sum_{j= \lfloor L \rfloor}^{\infty}\exp\left(  -\frac{j^2}{8t^2} \right)
      \leq \exp\left(- \frac{L^2}{8t^2} \right)
      + \int_L^{\infty}\exp \left( - \frac{x^2}{8t^2} \right)\d x\\
    & \leq \exp\left(- \frac{L^2}{8t^2} \right)
      +  \frac{1}{L}\int_L^{\infty}x\exp \left( - \frac{x^2}{8t^2} \right)\d x
    = \left( 1+ \frac{4t^2}{L} \right)\exp \left( - \frac{L^2}{8t^2}  \right) 
  \end{align*}
  which is bounded by $2\exp(-L)$ as soon as $L \geq 8 t^2$, thus implying our claim.
\end{proof}

We now bound the contribution of $\fund^-(L)$. 

\begin{lemma}\label{rMinusEstimate}
  With the notations of Lemma \ref{rPlusEstimate},
  \begin{align*}
    |R^-_K(X,\varepsilon,t,a,b,L)|
    \leq \frac{4 e}{r^2}
    \left( \frac{1}{ t^2} + \frac{b + 1}{r} \right)
    \left(1  + \frac{t}{r}  \right)
    \frac{\area(\fund^-(L))}{\area (X)}
    \left(1+ L e^L \right).
  \end{align*}
\end{lemma}

We observe that, due to the fact that the injectivity radius on $\fund^-(L)$ is not
bounded below by $L \gg 1$, we do not obtain an exponential decay like $e^{-L}$
in Lemma \ref{rPlusEstimate}. However, the ratio $\area(\fund^-(L)) / \area (X)$
will decay under the Benjamini--Schramm hypothesis.

\begin{proof}
  As before, we combine Proposition \ref{kernelEstimate} together with Lemma
  \ref{geodesicNumberBound} to obtain:
  \begin{align*}
    |R^-_K(X,\varepsilon,t,a,b,L)| 
    &\leq \frac{2e}{r^2}
      \left( \frac{1}{ t^2} + \frac{b + 1}{r} \right)
      \left(1  + \frac{t}{r}  \right)
      \frac{\area(\fund^-(L))}{\area (X)}
      \sum_{j\geq 0} \exp\left( j - \frac{j^2}{4t^2} \right).
  \end{align*}
  To deal with the last sum, we split it at $\lfloor 8t^2 \rfloor
  +1$. Proceeding as in Lemma \ref{rPlusEstimate}, we deduce:
  \begin{align*}
    \sum_{j\geq \lfloor 8t^2 \rfloor +1} \exp \left( j - \frac{j^2}{4t^2} \right)
    \leq
    \left( 1+ \frac{4t^2}{\lfloor 8t^2 \rfloor +1} \right)
    \exp \left( - \frac{(\lfloor 8t^2 \rfloor +1)^2}{8t^2} \right) \leq 2.
  \end{align*}
  For remaining indices we just bound naively:
  \begin{align*}
    \sum_{j=0}^{\lfloor 8t^2 \rfloor } \exp \left( j - \frac{j^2}{4t^2} \right)
    \leq
    \sum_{j=0}^{\lfloor 8t^2 \rfloor } \exp \left( j \right)
    \leq 
    8t^2\exp(8t^2)
    \leq
    L \exp(L)
  \end{align*}
  which leads to the claim.
\end{proof}

\subsection{Probabilistic kernel estimate}
\label{sec:prob-kern-estim}

The last step of this section is to use our probabilistic hypotheses, presented
in Section \ref{s:random_hyp}, to bound the kernel term. We prove the following.

\begin{proposition}\label{rEstimate}
  Let $0\leq a \leq b$, $g \geq 2$, and set
  $t := \frac{1}{4\sqrt{3}} \sqrt{\log g}$. Then for any
  $X\in \mathcal{A}_{g,k(g)}$, 
  \begin{align*}
    R_K(X,\varepsilon,t,a,b) = \mathcal{O}\left(\frac{b+1}{\sqrt{ \log g }} \right).
  \end{align*}
\end{proposition}

\begin{proof}
  Let us apply Lemmas \ref{rPlusEstimate} and \ref{rMinusEstimate} with the
  parameters  
  \begin{equation*}
    t := \frac{1}{4\sqrt{3}} \sqrt{\log g}
    \qquad
    L := 8 t^2 = \frac 16 \log(g)
    \qquad
    r := \frac{\sqrt{\log g}}{2 g^{\frac{1}{24}}} 
  \end{equation*}
  to a surface $X \in \mathcal{A}_{g,k(g)}$. By definition of the set
  $\mathcal{A}_{g,k(g)}$, the systole of $X$ is bounded below by $2r$.
  We observe that, for our choices of parameters,
  \begin{equation*}
    \frac{1}{r^2}
    \left( \frac{1}{ t^2} + \frac{b + 1}{r} \right)
    \left(1  + \frac{t}{r} \right)
    = \mathcal{O} \left(\frac{(b+1)t}{r^4} \right)
    = \mathcal{O} \left(\frac{b+1}{g^{- \frac 1 6} (\log
        g)^{\frac 32}} \right).
  \end{equation*}
  As a consequence, Lemma \ref{rPlusEstimate} implies
  \begin{equation*}
    R^+_K(X,\varepsilon,t,a,b,L)
    = \mathcal{O} \left(
      \frac{b+1}{g^{- \frac 1 6}(\log g)^{\frac 32}} \, e^{-L}\right)
    = \mathcal{O} \left(
      \frac{b+1}{(\log g)^{\frac 32}}\right).
  \end{equation*}
  Furthermore, by definition of $\mathcal{A}_{g,k(g)}$, 
  \begin{equation*}
    \frac{\area(\fund^-(L))}{\area(X)} \leq g^{- \frac 1 3}.
  \end{equation*}
  Lemma \ref{rMinusEstimate} then implies
  \begin{equation*}
    R^-_K(X,\varepsilon,t,a,b,L)
    = \mathcal{O} \left(
      \frac{b+1}{g^{- \frac 1 6} (\log g)^{\frac 32}} g^{- \frac 13}
      L e^L\right)
    = \mathcal{O} \left( \frac{b+1}{\sqrt{\log g}}\right)
  \end{equation*}
  which allows us to conclude because $R_K = R_K^+ + R_K^-$.  
\end{proof}

\section{Estimates for the number of eigenvalues}
\label{s:end_proof}

We are finally able to prove the main theorem. Throughout this section, we will
take the parameter $t$ to be equal to $\frac{\sqrt{\log g}}{4\sqrt{3}}$, for a
large genus $g$. Combining results from the last three subsections we obtain the
following:

\begin{lemma}\label{spectrumSumOEstimate}
  Let $g \geq 2$, $t = \frac{\sqrt{\log g}}{4\sqrt{3}}$, $0\leq a \leq b$. For any
  $X\in \mathcal{A}_{g,k(g)}$, any nontrivial spin structure $\varepsilon$ on $X$, 
\begin{align*}
  \frac{1}{\area (X)} \sum_{j=0}^{\infty} (h_t(\eig_j)+ h_t(-\eig_j))
  &=
    \frac{1}{4\pi}\int_a^b \eig \coth(\pi \eig)\d \eig
  + \mathcal{O}\left(\frac{b+1}{\sqrt{\log g}}  \right).
\end{align*}
\end{lemma}

\begin{proof}
  Rewrite formula $(\ref{traceSumFarFrom0})$ using Proposition
  \ref{integralEstimate}, \ref{rEstimate} and Remark
  \ref{rem:cuspsEstimate_applied}.
\end{proof}

From this we easily deduce Proposition \ref{eigenvaluesBigOEstimates}, the upper bound on the
counting function $N^{\Da}_{(X,\varepsilon)} (a,b)$ defined in Section \ref{sec:spectr-mult}.

\begin{proof}[Proof of Proposition \ref{eigenvaluesBigOEstimates}]
  First, we use the bound $0 \leq \eig \, \coth(\pi \eig) \leq \eig + 1/\pi$ for $\eig>0$ to obtain
  \begin{equation*}
    \int_a^b \eig \coth(\pi \eig)\d \eig
    \leq \int_a^b \left( \eig + \frac 1 \pi \right)\d \eig
    = \mathcal{O} \left( b^2-a^2 + b-a\right)
    = \mathcal{O} \left( (b+1) (b-a)\right)
  \end{equation*}
  because $b^2-a^2 = (b-a)(b+a)$ and $a \leq b$. 
  Then, Lemma \ref{spectrumSumOEstimate} implies
  \begin{equation}
    \label{eq:upper_bound_00}
    \frac{1}{\area (X)} \sum_{j=0}^{\infty} (h_t(\eig_j)+ h_t(-\eig_j))
     =
      \mathcal{O}\left( (b+1) \left( b-a  + \frac{1}{\sqrt{\log g}} \right) \right).
  \end{equation}
  We then observe that
  \begin{equation}
    \label{eq:N_vs_h_t_upper}
    N^{\Da}_{(X,\varepsilon)} (a,b) \inf_{\eig \in [a,b]} h_t(\eig)
    \leq
    \sum_{j=0}^{\infty} (h_t(\eig_j)+ h_t(-\eig_j))
  \end{equation}
  by positivity of $h_t$, and because $N^{\Da}_{(X,\varepsilon)}$ counts the
  number of indices $j$ such that $a \leq \eig_j \leq b$ by definition.
  Moreover, the restriction of the function $h_t$ to $[a,b]$ attains its infimum
  at both endpoints $a$ and $b$. Then, we consider the two following regimes.
  \begin{itemize}
  \item If $t(b - a)\geq 1$, by Lemma \ref{lem:bounds_test_function}.(\ref{htBound}) for either one
    of these two values,
    \begin{equation}
      \label{eq:lower_bound_ht_proof_upper}
      \inf_{\eig \in [a,b]} h_t(\eig)
       \geq
        \frac{1}{2}-\frac{\exp(-t^2(b-a)^2)}{2\sqrt{\pi} \, t(b-a)}
        \geq
        \frac{1}{2} - \frac{e^{-1}}{2\sqrt{\pi}}
        >
        \frac{1}{3}.
    \end{equation}
    Then equations \eqref{eq:upper_bound_00}, \eqref{eq:N_vs_h_t_upper} and
    \eqref{eq:lower_bound_ht_proof_upper} together imply our claim.
  \item If $t(b-a)\leq 1$, then we note that $b \leq a + 1/t$ and hence
    \begin{equation*}
      N^{\Da}_{(X,\varepsilon)} (a,b)
      \leq
        N^{\Da}_{(X,\varepsilon)} \left( a, a + \frac{1}{t} \right).
      \end{equation*}
      We apply the first case to the parameters $a$ and $b':= a+1/t$, and obtain
      \begin{equation*}
        \frac{N^{\Da}_{(X,\varepsilon)} (a,a + 1/t)}{\area (X)}
        = \mathcal{O}\left( \left(a + \frac 1 t+1 \right)
        \left(\frac 1t + \frac{1}{\sqrt{\log g}}
        \right) \right)
      = \mathcal{O} \left( \frac{a+1}{\sqrt{\log g}} \right),
    \end{equation*}
    which is enough to conclude because $a \leq b$.
  \end{itemize}
\end{proof}

We are now ready to conclude to the proof of our main result, Theorem \ref{eigenvaluesSharpEstimates}.
We prove the upper and lower bounds separately, because they rely on
a different method.

\begin{proof}[Proof of the upper bound of Theorem \ref{eigenvaluesSharpEstimates}]
  First, we note that if $t(b - a)< \sqrt{2e}$ then
  \begin{align*}
    \int_a^b \eig \coth(\pi \eig)\d \eig 
    = \mathcal{O}((b+1)(b - a))
    = \mathcal{O}\left( \frac{b +1 }{\sqrt{\log g}} \right),
  \end{align*}
  and hence the upper bound is then a trivial consequence of
  Proposition~\ref{eigenvaluesBigOEstimates}. Thus, for the rest of the proof, we shall assume that
  $t(b - a) \geq \sqrt{2e}$. The control of the function $h_t$ given by Lemma
  \ref{lem:bounds_test_function}.(\ref{htBound}) is not optimal near $a$ and $b$. Therefore we
  decompose the counting function as:
  \begin{align}
    \label{eq:decomp_N_fine}    
    N^{\Da}_{(X,\varepsilon)} (a,b) 
    = 
    N^{\Da}_{(X,\varepsilon)} \left(a , a + \eta \right) 
    + 
    N^{\Da}_{(X,\varepsilon)} \left(a + \eta, b - \eta \right) 
    + 
    N^{\Da}_{(X,\varepsilon)}  \left(b - \eta, b\right)
  \end{align}
  for a number $\eta \in [\frac 1t, \frac{b-a}{2}]$ that will be
  picked later (note that the hypothesis $t(b - a) \geq \sqrt{2e}$
  implies that this interval is not empty).
  
  On the one hand, we observe that the first and the third term on
  the right hand side of~\eqref{eq:decomp_N_fine} can easily be
  bounded above using Proposition~\ref{eigenvaluesBigOEstimates}:
  \begin{align}
    \label{eq:decomp_N_fine_bound_T1}    
    \frac{N^{\Da}_{(X,\varepsilon)} \left(a , a + \eta \right)
    + N^{\Da}_{(X,\varepsilon)}  \left(b - \eta, b\right)
    }{\area (X)}
    =
    \mathcal{O} \left( (b+1) \left(\eta + \frac{1}{\sqrt{\log g}} \right) \right)
    =
    \mathcal{O}\left( (b +1)\eta \right).
  \end{align}
  
  On the other hand, we can proceed as in the proof of Proposition
  \ref{eigenvaluesBigOEstimates} to bound the second term of the
  right hand side, except more finely this time.  More precisely, we
  write again
  \begin{equation*}
      \frac{N^{\Da}_{(X, \varepsilon)}\left(a + \eta, b - \eta
        \right)}{\area(X)} \inf_{\eig \in [a + \eta , b - \eta]} h_t(\eig)  \leq
      \sum_{j=0}^{\infty}\left( h_t(\eig_j)+h_t(-\eig_j)\right) 
  \end{equation*}
  and then use Lemma \ref{spectrumSumOEstimate} to obtain a constant
  $C>0$ such that
  \begin{equation}
    \label{eq:fine_upper_N}
    \frac{N^{\Da}_{(X, \varepsilon)}\left(a + \eta, b - \eta
      \right)}{\area(X)} \inf_{\eig\in [a + \eta , b - \eta]} h_t(\eig)
    \leq \frac{1}{4\pi}\int_a^b \eig \coth(\pi \eig)\d \eig + C \frac{b +
      1}{\sqrt{\log g}}.
  \end{equation}
  Let us estimate the infimum of $h_t$ on $[a+\eta, b-\eta]$. This infimum is attained at both
  endpoints $a+\eta$ and $b-\eta$. Using Lemma \ref{lem:bounds_test_function}.(\ref{htBound}) we
  obtain:
  \begin{align*}
    \inf_{\eig\in [a + \eta , b - \eta]} h_t(\eig) 
     \geq 
      1 - \frac{e^{-t^2\eta^2}}{2\sqrt{\pi} t \eta}
      - \frac{e^{-t^2 (b - a -\eta)^2}}{2 \sqrt{\pi} t (b - a -\eta)}
    \geq
      1 - \frac{e^{-t^2\eta^2}}{\sqrt{\pi}t\eta}
  \end{align*}
  because $b - a - \eta \leq \eta$. We now observe that, for all
  $0 \leq x \leq 1/2$, $(1-x)^{-1} \leq 1+2x$. We apply this inequality to
  $x = e^{-t^2\eta^2}/(\sqrt{\pi}t\eta) \leq 1/2$ (thanks to
  the fact that $t\eta \geq 1$) and get:
  \begin{equation}
    \label{eq:fine_bound_inf}
    \left( \inf_{\eig\in [a + \eta , b - \eta]} h_t(\eig) \right)^{-1}
    \leq 1 + 2 \frac{e^{-t^2\eta^2}}{\sqrt{\pi}t\eta}
    \leq 1 + 2 e^{-t^2\eta^2}.
  \end{equation}
  We now use the bound \eqref{eq:fine_bound_inf} into equation
  \eqref{eq:fine_upper_N}, which yields
  \begin{align*}
     \frac{N^{\Da}_{(X, \varepsilon)}\left(a + \eta, b - \eta
      \right)}{\area(X)} 
    & \leq (1 + 2 e^{-t^2\eta^2})
    \left( \frac{1}{4\pi}\int_a^b \eig \coth(\pi \eig)\d \eig + C \frac{b +
        1}{\sqrt{\log g}} \right) \\
    & \leq \frac{1}{4\pi}\int_a^b \eig \coth(\pi \eig)\d \eig
      + e^{-t^2\eta^2} \int_a^b \eig \coth(\pi \eig)\d \eig
      + 3 C \frac{b + 1}{\sqrt{\log g}}.
  \end{align*}
  By direct computations, one can check that the number
  \begin{equation}
    \label{eq:value_eta}
    \eta := \frac{1}{t}\sqrt{\log \left( \sqrt{\frac e 2}  \, t(b - a) \right)}  
  \end{equation}
  satisfies the hypothesis $\frac 1t \leq \eta \leq \frac{b-a}{2}$ thanks to the
  assumption $t(b - a)\geq \sqrt{2e}$. Then,
  \begin{equation*}
    e^{-t^2\eta^2} \int_a^b \eig \coth(\pi \eig)\d \eig
    = \mathcal{O} \left(  e^{-t^2\eta^2} (b+1)(b-a) \right)
    = \mathcal{O} \left(  \frac{b+1}{t} \right)
    = \mathcal{O} \left(  \frac{b+1}{\sqrt{\log(g)}} \right)
  \end{equation*}
  and therefore
  \begin{equation}
    \label{eq:decomp_N_fine_bound_T2}
    \frac{N^{\Da}_{(X, \varepsilon)}\left(a + \eta, b - \eta
      \right)}{\area(X)}
    \leq \frac{1}{4\pi}\int_a^b \eig \coth(\pi \eig)\d \eig
    + C' \frac{b+1}{\sqrt{\log g}}
  \end{equation}
  for a constant $C'$. We can then conclude using the
  decomposition~\eqref{eq:decomp_N_fine}, the bounds
  \eqref{eq:decomp_N_fine_bound_T2} and \eqref{eq:decomp_N_fine_bound_T1} with
  our value of $\eta$ specified in \eqref{eq:value_eta}.
\end{proof}

\begin{proof}[Proof of the lower bound of Theorem \ref{eigenvaluesSharpEstimates}]
  Since $0\leq h_t\leq 1$ everywhere one has:
  \begin{align*}
    N^{\Da}_{(X,\varepsilon)}(a,b)
    \geq 
    \sum_{a \leq \eig_j \leq b} h_t(\eig_j) 
  \end{align*}
  which we can rewrite as
  \begin{equation*}
    N^{\Da}_{(X,\varepsilon)}(a,b)
    \geq
    \sum_{j=0}^{\infty}\left( h_t(\eig_j) +h_t(-\eig_j)\right)
    - \left( \sum_{\eig_j < a} h_t(\eig_j)
      + \sum_{\eig_j> b} h_t(\eig_j)
      + \sum_{j=0}^{\infty} h_t(-\eig_j)\right).
  \end{equation*}
  By Lemma \ref{spectrumSumOEstimate}, there exists $C''>0$ such that:
  \begin{align*}
    \frac{1}{\area (X)}\sum_{j=0}^{\infty}
    \left( h_t(\eig_j) +h_t(-\eig_j)\right)
    \geq
    \frac{1}{4\pi}\int_a^b \eig \coth(\pi \eig)\d \eig - 
    C'' \frac{b + 1}{\sqrt{\log g}}.
  \end{align*}
  It therefore suffices to prove that the three sums we subtract are
  $\mathcal{O}\left( \frac{b+1}{\sqrt{\log g}} \right)$ to conclude.
  
  We shall only present the proof for the sum after eigenvalues
  larger than $b$, because one can treat the other cases
  similarly. For a non-negative integer $k$ denote
  $b_k := b + \frac{k}{t}$. Then,
  \begin{align*}
    \frac{1}{\area (X)}\sum_{\eig_j> b}  h_t(\eig_j)
    = \frac{1}{\area (X)} \sum_{k=0}^{\infty}
    \sum_{b_k \leq \eig_j < b_{k+1}} h_t(\eig_j)
    \leq \sum_{k=0}^{\infty}
    \frac{N^{\Da}_{(X,\varepsilon)}(b_k,b_{k+1})}{\area (X)}
    \sup_{\eig \in [b_k,b_{k+1}] }h_t(\eig). 
  \end{align*}
  For $k \geq 0$, since $b_{k+1}-b_k = 1/t = \mathcal{O}(1/\sqrt{\log(g)})$,
    Proposition \ref{eigenvaluesBigOEstimates} implies that
  \begin{equation*}
    \frac{N^{\Da}_{(X,\varepsilon)}(b_k,b_{k+1})}{\area (X)}
    = \mathcal{O} \left(
      \frac{b_{k+1}+1}{\sqrt{\log(g)}}\right)
    = \mathcal{O} \left(
      \frac{b+1}{\sqrt{\log(g)}}
      (k+1)\right).
  \end{equation*}
  We then apply Lemma \ref{lem:bounds_test_function}.(\ref{htBound}) to bound the supremum of $h_t$
  on $[b_k, b_{k+1}]$ for the terms $k \geq 1$, and obtain that
  \begin{align*}
    \sum_{k=0}^{\infty}
    (k+1) \sup_{\eig\in [b_k,b_{k+1}] }h_t(\eig)
    & \leq 1 + \sum_{k=1}^{\infty}
    (k+1) \, \frac{\exp(-t^2(b_k - b)^2)}{2\sqrt{\pi} t (b_k - b)} \\
    & \leq 1 + \sum_{k=1}^{\infty}
    (k+1) \, \frac{\exp(-k^2)}{2\sqrt{\pi} k} = \mathcal{O}(1)
  \end{align*}
  which is what we need to conclude.
\end{proof}

\bibliographystyle{plain}
\bibliography{bibliography}

\end{document}